\documentclass[10pt]{amsart}
\usepackage{latexsym}
\usepackage{amssymb}
\usepackage{amsmath}
\usepackage{mathrsfs}

\newtheorem{theorem}{Theorem}[section]

\newtheorem{corollary}[theorem]{Corollary}
\newtheorem{definition}[theorem]{Definition}

\newtheorem{lemma}[theorem]{Lemma}
\newtheorem{proposition}[theorem]{Proposition}

\begin{document}
\title[Nonstandard analysis and sumset phenomenon]{Nonstandard analysis and the sumset phenomenon in arbitrary amenable groups}
\author{Mauro Di Nasso}
\address{Dipartimento di Matematica, Universit\`{a} di Pisa, Italy.}
\email{dinasso@dm.unipi.it}
\author{Martino Lupini}
\address{Department of Mathematics and Statistics, York University, Canada.}
\email{mlupini@mathstat.yorku.ca}
\thanks{The first named author was supported by PRIN ``O-minimalit\`a,
metodi e modelli nonstandard e applicazioni''. The second author was
supported by the York University Elia Scholars Program and by the York
University Graduate Development Fund.}
\subjclass[2010]{03H05, 43A07, 11B05, 11B13.}
\keywords{Nonstandard analysis; Amenable group; Banach density;
piecewise syndetic set; product set; Delta-set.}

\begin{abstract}
Beiglb\"ock, Bergelson and Fish proved that if subsets $A$, $B$ of a
countable discrete amenable group $G$ have positive Banach densities $\alpha$
and $\beta$ respectively, then the product set $AB$ is piecewise syndetic,
i.e.\ there exists $k$ such that the union of $k$-many left translates of $%
AB $ is thick. Using nonstandard analysis we give a shorter alternative
proof of this result that does not require $G$ to be countable and moreover
yields the explicit bound $k\le 1/\alpha\beta$. We
also prove with similar methods that if $\{A_i\}_{i=1}^n$ are finitely many subsets of $G$ having
positive Banach densities $\alpha _i$ and $G$ is countable, then there
exists a subset $B$ whose Banach density is at least $\prod_{i=1}^{n}
\alpha_i $ and such that $BB^{-1}\subseteq\bigcap_{i=1}^{n}A_i A_i^{-1}$. In
particular, the latter set is piecewise Bohr.
\end{abstract}

\maketitle

\section*{Introduction.}

Using nonstandard analysis, in 2000 R. Jin proved that the sumset
$A+B$ of two sets of integers is
piecewise syndetic whenever both $A$ and $B$ have positive Banach density (%
\cite{jin}). Afterwards, with ergodic theory, M. Beiglb\"{o}ck, V. Bergelson and A. Fish
generalized Jin's theorem showing that if two subsets $A$ and $B$ of a \emph{%
countable} amenable group have positive Banach density, then their product
set $AB$ is piecewise syndetic, and in fact piecewise Bohr (\cite{bbf}). In
this paper, by using the nonstandard characterization of Banach density in
amenable groups, we extend that result to the uncountable case, and we also
provide an explicit bound on the number of left translates of $AB$ that are
needed to cover a thick set. Moreover, we extend some of the properties of
Delta-sets $A-A$ proved in \cite{dn} for sets of integers, to the general
setting of amenable groups. In particular, applying the pointwise ergodic
theorem for countable amenable groups in the nonstandard setting, we show
that any finite intersection $\bigcap_{i=1}^{n}A_{i}A_{i}^{-1}$ of sets $%
A_{i}$ of positive Banach density contains $BB^{-1}$ for some set $B$ of
positive Banach density and, as a consequence, is piecewise Bohr.

\smallskip Let us now introduce some terminology to be used in the paper, as
well as some combinatorial notions that we shall consider. Let $G$ be a
group. If $A,B\subseteq G$, we denote by $A^{-1}=\{a^{-1}\mid a\in A\}$ the
set of inverses of elements of $A$. A \emph{translate} of $A$ is a set of
the form $xA=\{xa\mid a\in A\}$. More generally, for subsets $A,B\subseteq G$%
, we denote the \emph{product set} $\{ab\mid a\in A\ \text{and}\ b\in B\}$
by $AB$.

A set $A\subseteq G$ is $k$-\emph{syndetic} if $k$ translates of $A$ suffice
to cover $G$, i.e.\ if $G=F A$ for some $F\subseteq A$ such that $|F|\leq k$%
. The set $A$ is \emph{thick} if the family of its translates $\{gA\mid g\in
G\}$ has the \emph{finite intersection property}, i.e.\ $\bigcap_{x\in H}xA
\neq \emptyset $ for all finite $H\subseteq G$ (equivalently, for every
finite $H$ there is $x\in G$ such that $Hx\subseteq A$).


Another relevant notion is obtained by combining syndeticity and thickness.
A set $A$ is \emph{piecewise} $k$-\emph{syndetic} if $k$ translates of $A$
suffice to cover a thick set, i.e.\ if $F A$ is thick for some $F\subseteq A$
such that $|F|\le k$. 
(For more on these notions see \cite{bhm}.)

\smallskip Familiarity will be assumed with the basics of \emph{nonstandard
analysis}, namely with the notions of \emph{hyperextension} (or nonstandard
extension), \emph{internal set}, \emph{hyperfinite set} (or ${}^{\ast }$%
finite set), the \emph{transfer principle}, and the properties of \emph{%
overspill} and $\kappa $-\emph{saturation} (recall that for every cardinal $%
\kappa $ there exist $\kappa$-saturated nonstandard models). Moreover, in
Section \ref{sec-countable} we shall also use the \emph{Loeb measure}. Good
references for the nonstandard notions used in this paper are e.g.\ the
introduction given in \cite{ck} \S 4.4, and the monograph \cite{Goldblatt}
where a comprehensive treatment of the theory is given. However, there are
several other interesting books on nonstandard analysis and its applications
that the reader may also want to consult (see e.g.\ \cite{nato,lnl,Cut}).

Let us now fix the \textquotedblleft nonstandard" notation we shall adopt
here. If $X$ is a standard entity, ${}^{\ast }\!{X}$ denotes its
hyperextension. A subset $A$ of ${}^{\ast }X$ is \emph{internal} if it
belongs to the hyperextension of the power set of $X$. If $\xi ,\zeta \in
{}^{\ast }\mathbb{R}$ are hyperreal numbers, we write $\xi \approx \zeta $
when $\xi $ and $\zeta $ are \emph{infinitely close}, i.e.\ when their
distance $|\xi -\zeta |$ is infinitesimal. If $\xi \in {}^{\ast }\mathbb{R}$
is finite, then its \emph{standard part} $\mathrm{st}(\xi )$ is the unique
real number which is infinitely close to $\xi $. We write $\xi \lesssim \eta
$ to mean that $\mathrm{st}(\xi -\eta )\leq 0$, i.e.\ $\xi <\eta $ or $\xi
\approx \eta $.

\smallskip We would like to thank Mathias Beiglb\"ock, Neil Hindman and
Renling Jin for many useful conversations, and Samuel Coskey for his
comments and suggestions.


\bigskip

\section{Amenable groups and Banach density}

In this paper we aim at generalizing combinatorial properties of sets of
integers related to their asymptotic density to more general groups. To this
purpose, it is convenient to work in the framework of amenable groups, that
are endowed with a suitable notion of density. Amenable groups admit several
equivalent characterizations (see e.g.\ \cite{Wagon,Paterson}). The most
convenient definition for our purposes is the following one, first isolated
by F\o lner \cite{Folner}.

\medskip

\begin{definition}
A group $G$ is \emph{amenable} if and only if it satisfies the following

\smallskip

\begin{itemize}
\item \emph{F\o lner's condition}: For every finite $H\subset G$ and for
every $\varepsilon >0$ there exists a finite set $K$ which is
\textquotedblleft $(H,\varepsilon )$-\emph{invariant}", i.e.\ $K$ is
nonempty and for every $h\in H$ one has
\begin{equation*}
\frac{|hK\bigtriangleup K|}{|K|}\ <\ \varepsilon .
\end{equation*}
\end{itemize}
\end{definition}

\medskip

The Banach density in amenable groups can be defined using the $%
(H,\varepsilon) $-invariant sets as the ``finite approximations" of $G$.
Using almost invariant sets ensures that the notion of density so obtained
is invariant by left translations.

\medskip

\begin{definition}
Let $G$ be an amenable group. The (upper) \emph{Banach density} $d(A)$ of a
subset $A\subseteq G$ is defined as the least upper bound of the set of real
numbers $\alpha $ such that for every finite $H\subseteq G$ and for every $%
\varepsilon >0$ there exists a finite $K$ which is $(H,\varepsilon )$%
-invariant and satisfies $\frac{|A\cap K|}{|K|}\geq \alpha $. Similarly, the
\emph{lower Banach density} $\underline{d}(A)$ is the least upper bound of
the numbers $\alpha $ such that, for some finite subset $H$ of $G$ and some $%
\varepsilon >0$, every finite subset $K$ of $G$ which is $\left(
H,\varepsilon \right) $-invariant satisfies $\frac{\left\vert A\cap
K\right\vert }{\left\vert K\right\vert }\geq \alpha $.
\end{definition}

\medskip It is not difficult to see that if $A$ is piecewise $k$-syndetic,
then $d\left( A\right) \geq 1/k$, and if $A$ is $k$-syndetic then $%
\underline{d}\left( A\right) \geq 1/k$.

\smallskip We now prove convenient nonstandard characterizations that will
be used in the sequel. The proof of the first part of Proposition \ref{prop:nonstandard-amenable} is essentially contained in Section 3 of \cite{Henson-measures} and reported here for convenience of the reader.

\medskip

\begin{proposition}\label{prop:nonstandard-amenable}
A group $G$ is amenable if and only if in every sufficiently saturated
nonstandard model one finds a ``\emph{F\o lner approximation}" of $G $, i.e.
a nonempty hyperfinite set $E\subseteq {}^{\ast}G$ such that for all $g\in G$%
:
\begin{equation*}
\frac{\left\vert gE\bigtriangleup E\right\vert}{\left\vert E\right\vert}
\approx 0.
\end{equation*}

Moreover, if $G$ is amenable, for all $A\subseteq G$ one has:
\begin{equation*}
d(A)\ =\ \max \left\{ \mathrm{st}\left( \frac{|{}^{\ast }\!{A}\cap E|}{|E|}%
\right) \,\Big|\,E\ \text{F\o lner approximation of }G\right\} .
\end{equation*}%
\begin{equation*}
\underline{d}(A)\ =\ \min \left\{ \mathrm{st}\left( \frac{|{}^{\ast }\!{A}%
\cap E|}{|E|}\right) \,\Big|\,E\ \text{F\o lner approximation of }G\right\} .
\end{equation*}
\end{proposition}

\begin{proof}
Assume first that $G$ is amenable. For $g\in G$ and $n\in\mathbb{N}$ let

\begin{equation*}
\Gamma (g,n)=\left\{ K\subseteq G\ \text{finite nonempty}\,\Big|\,\frac{%
\left\vert gK\bigtriangleup K\right\vert }{\left\vert K\right\vert }<\frac{1%
}{n}\right\} .
\end{equation*}

It is readily seen that by F\o lner's condition the family of all sets $%
\Gamma(g,n)$ has the finite intersection property. Then, in any nonstandard
model that satisfies $\kappa$-saturation with $\kappa>\max\{|G|,\aleph_0\}$,
the hyperextensions ${}^*\Gamma(g,n)$ have a nonempty intersection, and
every $E\in\bigcap\{{}^*\Gamma(g,n)\mid g\in G, n\in\mathbb{N}\}$ is a F\o %
lner approximation of $G$.

Conversely, given $H=\{g_1,\ldots,g_m\}\subseteq G$ and $\varepsilon>0$, the
existence of a nonempty finite $(H,\varepsilon)$-invariant set is proved by
applying transfer to the following property, which holds in the nonstandard
model: \emph{``There exists a nonempty hyperfinite $E\subseteq {}^{\ast }G$
such that $\left\vert g_iE\bigtriangleup E\right\vert< \varepsilon\left\vert
E\right\vert$ for all $i\in \{1,\ldots ,m\}$"}.

\smallskip Suppose now that $G$ is amenable and $A\subseteq G$. Consider a F%
\o lner approximation $E$ of $G$ and define $\alpha =\mathrm{st}\left( \frac{%
\left\vert {}^{\ast }\!A\cap E\right\vert }{\left\vert E\right\vert }\right)
$ . If $H=\left\{ g_{1},\ldots ,g_{n}\right\} \subseteq G$ and $\varepsilon
>0$, applying transfer to the statement \emph{\textquotedblleft There exists
a nonempty hyperfinite subset $E\subseteq {}^{\ast }G$ such that $\left\vert
g_{i}E\bigtriangleup E\right\vert <\varepsilon \left\vert E\right\vert $ for
every $i=1,2,\ldots ,n$ and $\left\vert {}^{\ast }\!A\cap E\right\vert
>(\alpha -\varepsilon )\left\vert E\right\vert $\textquotedblright } one
obtains the existence of an $\left( H,\varepsilon \right) $-invariant subset
$K$ of $G$ such that $\left\vert A\cap K\right\vert >(\alpha -\varepsilon
)\left\vert K\right\vert $. This shows that
\begin{equation*}
d\left( A\right) \geq \sup \left\{ \frac{\left\vert {}^{\ast }A\cap
E\right\vert }{\left\vert E\right\vert }\,\Big|\,E\text{ F\o lner
approximation of }G\right\} .
\end{equation*}%
It remains to show that the $\sup $ is a maximum, and it is equal to $%
d\left( A\right) $. Define for $g\in G$ and $n\in \mathbb{N}$,
\begin{equation*}
\Lambda _{A}\left( g,n\right) =\left\{ K\in \Gamma \left( g,n\right) \,\Big|%
\,\frac{\left\vert A\cap K\right\vert }{\left\vert K\right\vert }>d\left(
A\right) -\frac{1}{n}\right\} \text{.}
\end{equation*}%
It is easily seen that the family of all sets $\Lambda _{A}\left( g,n\right)
$ has the finite intersection property. As before, in any nonstandard model
that satisfies $\kappa $-saturation with $\kappa >\max \{|G|,\aleph _{0}\}$,
the hyperextensions ${}^{\ast }\Lambda _{A}\left( g,n\right) $ have a
nonempty intersection, and every $E\in \bigcap \{{}^{\ast }\Lambda
_{A}\left( g,n\right) \mid g\in G,n\in \mathbb{N}\}$ is a F\o lner
approximation of $G$ such that $\mathrm{st}\left( \frac{\left\vert {}^{\ast
}A\cap E\right\vert }{\left\vert E\right\vert }\right) =d\left( A\right) $.

\smallskip The proof of the nonstandard characterization of the lower Banach
density is similar and is omitted. 
\end{proof}

\medskip It is often used in the literature the notion of \emph{F\o lner
sequence}, i.e.\ a sequence $\left( F_{n}\right) _{n\in \mathbb{N}}$ of
finite subsets of $G$ such that, for all $g\in G$,
\begin{equation*}
\lim_{n\rightarrow \infty }\frac{|gF_{n}\bigtriangleup F_{n}|}{|F_{n}|}\ =\
0.
\end{equation*}

We remark that, if the above condition holds, then for every finite $%
H\subset G$ and for every $\varepsilon>0$ the sets $F_n$ are $%
(H,\varepsilon) $-invariant for all sufficiently large $n$. It follows that
a \emph{countable} group $G$ is amenable if and only if it admits a F\o lner
sequence. Moreover, in the countable case the F\o lner density of a set $%
A\subseteq G$ is characterized as follows:

\begin{equation*}
d(A)\ =\ \sup \left\{ \limsup_{n\rightarrow \infty }\frac{|A\cap F_{n}|}{
|F_{n}|}\,\Big|\,\left( F_{n}\right) _{n\in \mathbb{N}}\ \text{a F\o lner
sequence}\right\} .
\end{equation*}

\medskip It is a well known fact (see for example \cite{bbf}, Remark 1.1),
that if $\left( F_n\right)_{n\in \mathbb{N}}$ is any F\o lner sequence and $%
A\subseteq G$, then there is a sequence $\left( g_n\right) _{n\in \mathbb{N}
} $ of elements of $G$ such that
\begin{equation*}
d(A)=\limsup _{n\in \mathbb{N}}\frac{|A\cap F_{n}g_{n}|}{|F_n|}
\end{equation*}
From this, it immediately follows that, when $G=\mathbb{Z}$, the Banach
density as defined here coincides with the usual notion of Banach density
for sets of integers.

\smallskip For an extensive treatment of Banach density and its
generalizations in the context of semigroups, the reader is referred to \cite%
{hs2006}.

\medskip The following notion of density Delta-sets is a generalization of
the Delta-sets $A-A=\{a-a^{\prime}\mid a,a^{\prime}\in A\}$ of sets of
integers.

\medskip

\begin{definition}
Let $G$ be an amenable group, and let $\varepsilon \geq 0$. For $A\subseteq
G $, the corresponding \emph{$\varepsilon$-density Delta-set} (or \emph{$%
\varepsilon$-Delta-set} for short) is defined as $\Delta _{\varepsilon
}(A)=\{g\in G\mid d(A\cap gA)>\varepsilon \}$.
\end{definition}

\medskip

Observe that $\Delta_\varepsilon(A)\subseteq \Delta_0(A)\subseteq A A^{-1}$.
We now introduce a notion of embeddability between sets of a group. The idea
is to have a suitable partial ordering relation at hand that preserves the
finite combinatorial structure of sets.

\medskip

\begin{definition}
Let $A,B\subseteq G$. We say that $A$ is \emph{finitely embeddable} in $B$,
and write $A\vartriangleleft B$, if every finite subset of $A$ has a right
translate contained in $B$.
\end{definition}


\medskip It is immediate from the definitions that $A$ is thick if and only
if $G\vartriangleleft A$. Finite embeddability admits the following
nonstandard characterization.

\medskip

\begin{proposition}
Let $A,B\subseteq G$. Then $A\vartriangleleft B$ if and only if in every
sufficiently saturated nonstandard model one has $A\eta\subseteq {}^{\ast }B$
for some $\eta\in {}^{\ast }G$.
\end{proposition}

\begin{proof}
Notice that $A\vartriangleleft B$ if and only if the family $\{a^{-1}B\mid
\,a\in A\}$ has the finite intersection property. So, in any nonstandard
model that satisfies $\kappa$-\emph{saturation} with $\kappa>|A|$, the
intersection $\bigcap_{a\in A}a^{-1} {}^{\ast }B$ is nonempty. If $\eta$ is
an element of this set, then $A\eta\subseteq {}^{\ast }B$. Conversely,
suppose that $A\eta\subseteq {}^{\ast }B$ for some $\eta\in {}^{\ast }G$. If
$H=\{h_1,\ldots,h_n\}$ is a finite subset of $A$, one obtains the existence
of an element $x\in G$ such that $Hx\subseteq B$ by \emph{transfer} from the
statement: \emph{\ \textquotedblleft There exists $\eta\in {}^{\ast }G$ such
that $h_i\eta\in {}^{\ast }B$ for $i=1,\ldots,n$ \textquotedblright}. This
shows that $A\vartriangleleft B$ .
\end{proof}

\medskip It is easily verified that, if $A\vartriangleleft B$, then $%
AA^{-1}\subseteq BB^{-1}$ and $\Delta _{\varepsilon }\left( A\right)
\subseteq \Delta _{\varepsilon }\left( B\right) $ for every $\varepsilon
\geq 0$.

\bigskip

\section{Combinatorial properties in a nonstandard setting}

In this section we prove combinatorial properties in a nonstandard framework
that will be used as key ingredients in the proofs of our main results. The
first one below can be seen as a form of \emph{pigeonhole principle} that
holds in a hyperfinite setting.

\medskip

\begin{lemma}
Let $E$ be a hyperfinite set, and let $\{C_{\lambda }\mid\lambda\in
\Lambda\} $ be a finite family of internal subsets of $E$. Assume that $%
\gamma,\varepsilon$ are non-negative real numbers such that

\smallskip

\begin{itemize}
\item $\varepsilon<\gamma^2$\,;

\smallskip

\item $\mathrm{st}\left( \frac{|C_{\lambda }|}{|E|}\right) \geq \gamma $ for
every $\lambda \in \Lambda $\thinspace ;

\smallskip

\item $|\Lambda|>\frac{\gamma-\varepsilon}{\gamma^{2}-\varepsilon}$.
\end{itemize}

\smallskip \noindent Then there exist distinct $\lambda $, $\mu \in \Lambda $
such that
\begin{equation*}
\mathrm{st}\left( \frac{|C_{\lambda }\cap C_{\mu }|}{|E|}\right)
>\varepsilon \text{.}
\end{equation*}
\end{lemma}

\begin{proof}
Without loss of generality, let us assume that $\mathrm{st}(|C_{\lambda
}|/|E|)=\gamma $ for every $\lambda \in \Lambda $. Suppose by contradiction
that for all distinct $\lambda \neq \mu $:
\begin{equation*}
\mathrm{st}\left( \frac{|C_{\lambda }\cap C_{\mu }|}{|E|}\right) \leq
\varepsilon \text{.}
\end{equation*}

For $i\in E$, set $a_{i}=\sum_{i\in \Lambda }\chi _{\lambda }(i)$ where $%
\chi _{\lambda }:E\rightarrow \{0,1\}$ denotes the characteristic function
of $C_{\lambda }$. Observe that
\begin{equation*}
\sum_{i\in E}a_{i}=\sum_{\lambda \in \Lambda }\left\vert C_{\lambda
}\right\vert
\end{equation*}
and
\begin{equation*}
\sum_{i\in E}a_{i}^{2}=\sum_{\lambda \in \Lambda }\left\vert C_{\lambda
}\right\vert +\sum_{\lambda \neq \mu }\left\vert C_{\lambda }\cap C_{\mu
}\right\vert \text{.}
\end{equation*}

If we set $b_{i}=1$, then by the \emph{Cauchy-Schwartz inequality},
\begin{eqnarray*}
\left( \sum_{\lambda \in \Lambda }\left\vert C_{\lambda }\right\vert \right)
^{2} &=&\left( \sum_{i\in E}a_{i}b_{i}\right) ^{2} \\
&\leq &\left( \sum_{i\in E}a_{i}^{2}\right) \left( \sum_{i\in
E}b_{i}^{2}\right) \\
&=&\left\vert E\right\vert \left( \sum_{\lambda \in \Lambda }\left\vert
C_{\lambda }\right\vert +\sum_{\lambda \neq \mu }\left\vert C_{\lambda }\cap
C_{\mu }\right\vert \right) \text{.}
\end{eqnarray*}
Dividing by $|E|^{2}$, one gets
\begin{eqnarray*}
\left\vert \Lambda \right\vert ^{2}\gamma ^{2} &\approx &\left(
\sum_{\lambda \in \Lambda }\frac{|C_{\lambda }|}{|E|}\right) ^{2} \\
&\leq &\sum_{\lambda \in \Lambda }\frac{|C_{\lambda }|}{|E|}+\sum_{\lambda
\neq \mu }\frac{|C_{\lambda }\cap C_{\mu }|}{|E|} \\
&\approx &\left\vert \Lambda \right\vert \gamma +\sum_{\lambda \neq \mu }
\frac{|C_{\lambda }\cap C_{\mu }|}{|E|}.
\end{eqnarray*}
As there are $\left\vert \Lambda \right\vert \left( \left\vert \Lambda
\right\vert -1\right) $ ordered pairs $\left( \lambda ,\mu \right) $ such
that $\lambda \neq \mu $, we get
\begin{equation*}
\varepsilon |\Lambda |(|\Lambda |-1)\gtrsim \sum_{\lambda \neq \mu }\frac{
|C_{\lambda }\cap C_{\mu }|}{|E|}\gtrsim |\Lambda |\gamma (|\Lambda |\gamma
-1).
\end{equation*}
Dividing by $|\Lambda |$, we obtain that $\left\vert \Lambda \right\vert
\gamma ^{2}\leq \gamma +\varepsilon \left( \left\vert \Lambda \right\vert
-1\right) $, and finally:
\begin{equation*}
\left\vert \Lambda \right\vert \leq \frac{\gamma -\varepsilon }{\gamma
^{2}-\varepsilon }\text{.}
\end{equation*}
This contradicts our assumptions and concludes the proof.
\end{proof}

\medskip Recall that we called \emph{F\o lner approximation} of $G$ any
nonempty hyperfinite set $E\subseteq {}^{\ast }G$ such that for all $g\in G$
:
\begin{equation*}
\frac{\left\vert gE\bigtriangleup E\right\vert }{\left\vert E\right\vert }
\approx 0\text{.}
\end{equation*}

\smallskip

\begin{lemma}
\label{Lemma: Density delta set}

Let $E$ be a F\o lner approximation of $G$, and suppose that $C$ is an
internal subset of ${}^{\ast }G$ such that $\mathrm{st}\left( \frac{%
\left\vert C\cap E\right\vert }{\left\vert E\right\vert }\right) =\gamma >0$%
. Let $0\leq \varepsilon <\gamma ^{2}$ and $k=\left\lfloor \frac{\gamma
-\varepsilon }{\gamma ^{2}-\varepsilon }\right\rfloor $. Define
\begin{equation*}
\mathcal{D}_{\varepsilon }^{E}\left( C\right) =\left\{ g\in G\left\vert
\,st\left( \frac{\left\vert C\cap gC\cap E\right\vert }{\left\vert
E\right\vert }\right) >\varepsilon \right. \right\} .
\end{equation*}%
Then, for every $P\subseteq G$ and every $g_{0}\in P$ there exists $%
F\subseteq P$ such that $g_{0}\in F$, $\left\vert F\right\vert \leq k$ and $%
P\subseteq F\cdot \mathcal{D}_{\varepsilon }^{E}\left( C\right) $.
\end{lemma}

\begin{proof}
We define elements $g_{i}$ of $P$ by recursion. Suppose that $g_{i}$ has
been defined for $0\leq i<n$. If $P\subseteq \{g_{0},\ldots ,g_{n-1}\}\cdot
\mathcal{D}_{\varepsilon }^{E}(C)$, then set $g_{n}=g_{n-1}$. Otherwise,
pick
\begin{equation*}
g_{n}\in P\setminus \left( \{g_{0},\ldots ,g_{n-1}\}\cdot \mathcal{D}%
_{\varepsilon }^{E}(C)\right) .
\end{equation*}%
We claim that, $g_{k}=g_{k-1}$, i.e.\ $P\subseteq \{g_{0},\ldots
,g_{k-1}\}\cdot \mathcal{D}_{\varepsilon }^{E}(C)$. Suppose by contradiction
that this is not the case. Then, for every $i<j<k$, we have
\begin{equation*}
g_{j}\notin \{g_{0},\ldots ,g_{i}\}\cdot \mathcal{D}_{\varepsilon
}^{E}\left( C\right) \text{.}
\end{equation*}%
This implies that $g_{i}^{-1}g_{j}\notin \mathcal{D}_{\varepsilon }^{E}(C)$
and
\begin{equation*}
\varepsilon \ \geq \ \text{\textrm{st}}\left( \frac{\left\vert C\cap
g_{i}^{-1}g_{j}C\cap E\right\vert }{\left\vert E\right\vert }\right) \ =\
\text{\textrm{st}}\left( \frac{\left\vert g_{i}C\cap g_{j}C\cap E\right\vert
}{\left\vert E\right\vert }\right) .
\end{equation*}%
By the previous lemma applied to the family $\{g_{i}C\cap E\mid i<k\}$,
there exist $i<j<k$ such that
\begin{equation*}
\frac{\left\vert g_{i}C\cap g_{j}C\cap E\right\vert }{\left\vert
E\right\vert }>\varepsilon .
\end{equation*}%
This is a contradiction.
\end{proof}

\medskip

\begin{lemma}
\label{Lemma: concentration0} Let $U,V\subseteq {}^{\ast }G$ be hyperfinite
sets, and let $C\subseteq U$ and $D\subseteq V$ be internal subsets. Then
there exists $\zeta,\vartheta\in U$ such that
\begin{equation*}
(1)\quad\quad \frac{|D\zeta\cap C|}{|V|}\ \geq\ \frac{|C|}{|U|}\!\cdot\!
\frac{|D|}{|V|}- \max_{d\in D}\frac{|d U\bigtriangleup U|}{|U|}.
\end{equation*}
\begin{equation*}
(2)\quad\quad \frac{|\vartheta D\cap C|}{|V|}\ \geq\ \frac{|C|}{|U|}%
\!\cdot\! \frac{|D|}{|V|}- \max_{d\in D}\frac{|Ud\bigtriangleup U|}{|U|}.
\end{equation*}
\end{lemma}

\begin{proof}
Let $\chi _{C}:U\rightarrow \{0,1\}$ be the characteristic function of $C$.
For $d\in D$, one has
\begin{equation*}
\frac{1}{|U|}\sum_{u\in U}\chi _{C}(du)\ =\ \frac{|C\cap dU|}{|U|}\ =\ \frac{
|C|-|C\cap (U\setminus dU)|}{|U|}\ \geq \ \frac{|C|}{|U|}-\frac{
|dU\bigtriangleup U|}{|U|}.
\end{equation*}
Then,
\begin{eqnarray*}
\frac{1}{|U|}\sum_{u\in U}\frac{|Du\cap C|}{|V|} &=&\frac{1}{|U|}\sum_{u\in
U}\left( \frac{1}{|V|}\sum_{d\in D}\chi _{C}(du)\right) \\
&=&\ \frac{1}{|V|}\sum_{d\in D}\left( \frac{1}{|U|}\sum_{u\in U}\chi
_{C}(du)\right) \\
&\geq &\frac{1}{|V|}\sum_{d\in D}\left( \frac{|C|}{|U|}-\frac{
|dU\bigtriangleup U|}{|U|}\right) \\
&\geq &\frac{|C|}{|U|}\!\cdot \!\frac{|D|}{|V|}-\max_{d\in
D}|dU\bigtriangleup U|/|U| .
\end{eqnarray*}
Thus for some $\zeta \in U$,
\begin{equation*}
\frac{|D\zeta \cap C|}{|V|}\ \geq \ \frac{|C|}{|U|}\!\cdot \!\frac{|D|}{|V|}
-\max_{d\in D}\frac{|dU\bigtriangleup U|}{|U|}.
\end{equation*}

The second part of the statement is obtained applying the first part to the
opposite group of $G$ (which is amenable as well).
\end{proof}

\medskip

\begin{lemma}
\label{concentration} Let $G$ be an amenable group. Suppose $%
A_{0},A_{1},\ldots ,A_{n}\subseteq G$ are subsets with Banach densities $%
d(A_{i})\geq \alpha _{i}$. Then in every sufficiently saturated nonstandard
model there exist F\o lner approximations $E,F\subseteq {}^{\ast }G$ and
elements $\xi _{1},\ldots ,\xi _{n},\eta _{1},\ldots ,\eta _{n}\in {}^{\ast
}G$ such that
\begin{equation}
\quad \quad \quad \frac{|{}^{\ast }\!A_{0}\,\cap \,\left(
\bigcap_{i=1}^{n}{}^{\ast }\!A_{i}\,\xi _{i}\right) \,\cap \,E|}{|E|}\
\gtrsim \ \prod_{i=0}^{n}\alpha _{i}.  \label{Eq: concentration 1}
\end{equation}%
\begin{equation}
\quad \quad \frac{|{}^{\ast }\!A_{0}\,\cap \,\left( \bigcap_{i=1}^{n}\eta
_{i}\,{}^{\ast }\!A_{i}^{-1}\right) \,\cap \,F|}{|F|}\ \gtrsim \
\prod_{i=0}^{n}\alpha _{i}.  \label{Eq: concentration 2}
\end{equation}
\end{lemma}

\begin{proof}
We proceed by induction. Let us start with property $(1)$. The base $n=0$ is
given by the nonstandard characterization of Banach density. Now let the
subsets $A_{0},A_{1},\ldots ,A_{n+1}\subseteq G$ be given where $%
d(A_{i})\geq \alpha _{i}$. By the inductive hypothesis there exists a F\o %
lner approximation $V\subseteq {}^{\ast }G$ and elements $\xi _{1},\ldots
,\xi _{n}\in {}^{\ast }G$ that satisfy $|{}^{\ast }\!A_{0}\cap
(\bigcap_{i=1}^{n}{}^{\ast }\!A_{i}\,\xi _{i})\cap V|/|V|\gtrsim
\prod_{i=0}^{n}\alpha _{i}$. We now want to find a F\o lner approximation $U$
that witnesses $d(A_{n+1})\geq \alpha _{n+1}$ and with the additional
feature of being \textquotedblleft almost invariant" with respect to left
translates by elements in $V$. To this purpose, pick a hyperfinite $%
V^{\prime }\supseteq V\cup G$ (notice that this is possible by $\kappa $%
-saturation with $\kappa >|G|$). Consider the following property that
directly follows from the definition of F\o lner density: \textquotedblleft
\emph{For every $k\in \mathbb{N}$ and every finite $H\subseteq G$ there
exists a nonempty finite $K\subseteq G$ which is $(H,1/k)$-invariant and
such that the relative density $|A_{n+1}\cap K|/|K|>\alpha _{n+1}-1/k$}%
\textquotedblright\ . If $\nu \in {}^{\ast }\mathbb{N}$, by \emph{transfer}
we get a nonempty hyperfinite $U\subseteq {}^{\ast }G$ that is $\left(
V^{\prime },\frac{1}{\nu }\right) $-invariant (and, in particular, is a F\o %
lner approximation of $G$) and such that
\begin{equation*}
\frac{|{}^{\ast }\!A_{n+1}\cap U|}{|U|}>\alpha _{n+1}-\frac{1}{\nu }\approx
\alpha _{n+1}\text{.}
\end{equation*}%
Now apply $(1)$ of the previous lemma to the internal sets $C={}^{\ast
}\!A_{n+1}\cap U\subseteq U$ and $D={}^{\ast }\!A_{0}\cap
(\bigcap_{i=1}^{n}{}^{\ast }\!A_{i}\,\xi _{i})\cap V\subseteq V$, and pick
an element $\zeta \in U$ such that
\begin{equation*}
\frac{|D\zeta \cap C|}{|V|}\ \geq \ \frac{|C|}{|U|}\!\cdot \!\frac{|D|}{|V|}%
-\max_{d\in D}\frac{|dU\bigtriangleup U|}{|U|}\ \geq \ \frac{|C|}{|U|}%
\!\cdot \!\frac{|D|}{|V|}-\frac{1}{\nu }\ \gtrsim \ \prod_{i=0}^{n+1}\alpha
_{i}.
\end{equation*}%
This yields the conclusion with $E=V$. In fact, by letting $\xi _{n+1}=\zeta
^{-1}$
\begin{equation*}
\frac{|D\zeta \cap C|}{\left\vert V\right\vert }\leq \frac{|D\zeta \cap
{}^{\ast }\!A_{n+1}|}{\left\vert V\right\vert }=\frac{|D\cap {}^{\ast
}\!A_{n+1}\xi _{n+1}|}{\left\vert V\right\vert }=\frac{\left\vert {}^{\ast
}\!A_{0}\,\cap \,\left( \bigcap_{i=1}^{n+1}{}^{\ast }\!A_{i}\,\xi
_{i}\right) \,\cap \,V\right\vert }{\left\vert V\right\vert }.
\end{equation*}

As for $(2)$, we proceed in a similar way as above by considering sets of
inverses. Precisely, let $V\subseteq {}^{\ast }G$ be a F\o lner
approximation of $G$ and $\eta _{1},\ldots ,\eta _{n}$ be elements of $%
{}^{\ast }G$ that satisfy
\begin{equation*}
\frac{|{}^{\ast }\!A_{0}\,\cap \,\left( \bigcap_{i=1}^{n}\eta _{i}\,{}^{\ast
}\!A_{i}^{-1}\right) \,\cap \,V|}{|V|}\ \gtrsim \ \prod_{i=0}^{n}\alpha _{i}.
\end{equation*}
Pick a F\o lner approximation $U$ that witnesses $d(A_{n+1})\geq \alpha
_{n+1}$ and with the additional feature of being \textquotedblleft almost
invariant" with respect to translates by elements in the set of inverses $%
V^{-1}$, i.e.\ $|{}^{\ast }\!A_{n+1}\cap U|/|U|\gtrsim \alpha _{n+1}$ and $%
|xU\bigtriangleup U|/|U|\approx 0$ for all $x\in V^{-1}$. Then apply $(2)$
of the previous lemma to the internal sets $C={}^{\ast }\!A_{n+1}^{-1}\cap
U^{-1}\subseteq U^{-1}$ and $D={}^{\ast }\!A_{0}\cap (\bigcap_{i=1}^{n}\eta
_{i}\,{}^{\ast }\!A_{i}^{-1})\cap V\subseteq V$, and get the existence of an
element $\vartheta \in U^{-1}$ such that
\begin{eqnarray*}
\frac{|\vartheta D\cap C|}{|V|}\ &\geq &\ \frac{|C|}{|U^{-1}|}\!\cdot \!
\frac{|D|}{|V|}-\max_{d\in D}\frac{|U^{-1}d\bigtriangleup U^{-1}|}{|U^{-1}|}
\\
&=&\ \frac{|C^{-1}|}{|U|}\!\cdot \!\frac{|D|}{|V|}-\max_{d\in D^{-1}}\frac{
|dU\bigtriangleup U|}{|U|} \\
&\gtrsim &\prod_{k=1}^{n+1}\alpha _{k}\text{.}
\end{eqnarray*}
Since
\begin{equation*}
\frac{|\vartheta D\cap C|}{\left\vert V\right\vert }\ \leq \ \frac{
|\vartheta D\cap {}^{\ast }\!A ^{-1}_{n+1}|}{\left\vert V\right\vert }\ =\
\frac{|D\cap \vartheta ^{-1}{}^{\ast }\!A ^{-1}_{n+1}|}{\left\vert
V\right\vert }\ =\ \frac{\left\vert {}^{\ast }\!A_{0}\cap \left(
\bigcap_{i=0}^{n+1}\eta _{i}{}^{\ast }\!A ^{-1}_{n+1}\right) \cap
V\right\vert }{\left\vert V\right\vert }\text{,}
\end{equation*}
the statement is proved by letting $F=V$ and $\eta _{n+1}=\vartheta ^{-1}$.
\end{proof}

\bigskip

\section{Intersection properties of Delta-sets and Jin's theorem}

The nonstandard lemmas of the previous section entail a general result about
intersections of density Delta-sets.

\medskip

\begin{theorem}
\label{Theorem: syndetic delta sets} Suppose that, for $i\leq n$, $A_{i}$ is
a subset of $G$ of positive Banach density $\alpha _{i}$. Let $0\leq
\varepsilon <\beta ^{2}$ where $\beta =\prod_{i=1}^{n}\alpha _{i}$, $%
P\subseteq G$ and $g_{0}\in P$. If $r=\left\lfloor {\frac{\beta -\varepsilon
}{\beta ^{2}-\varepsilon }}\right\rfloor $, then there exists a finite $%
L\subseteq G$ such that $\left\vert L\right\vert \leq r$, $g_{0}\in L$ and $%
P\subseteq L\cdot \left( \bigcap_{i=1}^{n}\Delta _{\varepsilon
}(A_{i})\right) $.
\end{theorem}

\begin{proof}
By Lemma \ref{concentration} where $A_{0}=G$, we can pick a F\o lner
approximation $E\subseteq {}^{\ast }G$ and elements $\xi _{1},\ldots ,\xi
_{n}\in {}^{\ast }G$ such that
\begin{equation*}
\frac{\left\vert (\bigcap_{i=1}^{n}{}{}^{\ast }\!A_{i}\xi _{i})\cap
E\right\vert }{|E|}\ \gtrsim \ \beta .
\end{equation*}%
Define the internal set $C=\bigcap_{i=1}^{n}{}{}^{\ast }\!A_{i}\xi _{i}$ and
observe that
\begin{equation*}
\mathcal{D}_{\varepsilon }^{E}(C)\ \subseteq \ \bigcap_{i=1}^{n}\Delta
_{\varepsilon }({}^{\ast }\!A_{i}).
\end{equation*}%
To see this, notice that if $g\in \mathcal{D}_{\varepsilon }^{E}(C)$ then
for every $j=1,\ldots ,n$:
\begin{eqnarray*}
\varepsilon  &<&\frac{|C\cap gC\cap E|}{|E|}\ =\ \frac{\left\vert \left(
\bigcap_{i=1}^{n}{}^{\ast }\!A_{i}\xi _{i}\right) \cap g\left(
\bigcap_{i=1}^{n}{}^{\ast }\!A_{i}\xi _{i}\right) \cap E\right\vert }{|E|} \\
&\leq &\frac{\left\vert {}^{\ast }\!A_{j}\xi _{j}\cap g\,{}^{\ast
}\!A_{j}\xi _{j}\cap E\right\vert }{|E|}\ =\ \frac{|{}^{\ast }(A_{j}\cap
gA_{j})\cap E\xi _{j}^{-1}|}{|E\xi _{j}^{-1}|}\text{.}
\end{eqnarray*}%
Now apply Lemma \ref{Lemma: Density delta set} to $C$ and get a finite $%
L\subseteq P$ such that $|L|\leq r$, $g_{0}\in L$ and
\begin{equation*}
P\subseteq L\cdot \mathcal{D}_{\varepsilon }^{E}(C)\subset L\cdot
\bigcap_{i=1}^{n}\Delta _{\varepsilon }\left( {}^{\ast }\!A_{i}\right) \text{%
.}\qedhere
\end{equation*}
\end{proof}

\medskip By applying Theorem \ref{Theorem: syndetic delta sets} where $P=G$,
one immediately obtains the following

\medskip

\begin{corollary}
Under the assumptions of Theorem \ref{Theorem: syndetic delta sets}, $%
\bigcap_{i=1}^{n}\Delta _{\varepsilon }(A_{i})$ is $r$-syndetic and, as a
consequence, its lower Banach density is at least $1/r$.
\end{corollary}

\medskip For $k\in\mathbb{N}$, denote by

\smallskip

\begin{itemize}
\item $A^{(k)}=\{g^{k}\mid g\in A\}$ the set of $k$-th powers of elements of
$A$.

\smallskip

\item $\sqrt[k]{A}=\{g\in G\mid g^k\in A\}$ the set of $k$-th roots of
elements of $A$.
\end{itemize}

\medskip

\begin{corollary}
Under the assumptions of Theorem \ref{Theorem: syndetic delta sets}, for
every $k\in\mathbb{N}$ the intersection $\bigcap_{i=1}^n\sqrt[k]{
\Delta_\varepsilon(A_i)}$ is $r$-syndetic and, as a consequence, its lower
Banach density is at least $1/r$.
\end{corollary}

\begin{proof}
Apply Theorem \ref{Theorem: syndetic delta sets} with $P=G^{(k)}$ and get
the existence of a finite set $L\subseteq G^{(k)}$ such that $|L|\leq \frac{
\beta -\varepsilon }{\beta ^{2}-\varepsilon }$ and $G^{(k)}\subseteq L\cdot
\left( \bigcap_{i=1}^{n}\Delta _{\varepsilon }(A_{i})\right) $. Pick $%
H\subseteq G$ such that $|H|=|L|$ and $H^{(k)}=L$. Then for every $g\in G$
one has $g^{k}=h^{k}\cdot x$ for suitable $h\in H$ and $x\in
\bigcap_{i=1}^{n}\Delta _{\varepsilon }(A_{i})$. Equivalently, for every $%
g\in G$ there exists $h\in H$ such that $(h^{-1}g)^{k}\in
\bigcap_{i=1}^{n}\Delta _{\varepsilon }(A_{i})$, and hence
\begin{equation*}
h^{-1}g\in \bigcap_{i=1}^{n}\sqrt[k]{\Delta _{\varepsilon }\left(
A_{i}\right) }\text{.}
\end{equation*}
This shows that $G=H\cdot \left( \bigcap_{i=1}^{n}\sqrt[k]{\Delta
_{\varepsilon }(A_{i})}\right) $.
\end{proof}

\medskip Next, we prove the existence of an explicit bound in Jin's theorem
that only depends on the densities of the given sets.

\medskip

\begin{theorem}
\label{Theorem: Jin with bound} Let $G$ be an amenable group. If $X\subseteq
G$ is infinite, $w\in X$ and $A,B\subseteq G$ have positive Banach densities
$d(A)=\alpha$ and $d(B)=\beta$ respectively, then there exists a finite $%
F\subset X$ such that:

\smallskip

\begin{itemize}
\item $w\in F$;

\smallskip

\item $|F|\leq\frac{1}{\alpha\beta}$;

\smallskip

\item $X\vartriangleleft F A B$.
\end{itemize}
\end{theorem}

\begin{proof}
By Lemma \ref{concentration}, we can pick a F\o lner approximation $%
E\subseteq {}^{\ast }G$ and an element $\eta \in {}^{\ast }G$ such that the
internal set $X={}^{\ast }\!A \cap \eta {}^{\ast }\!B^{-1}\cap E$ has
relative density $|X|/|E|\gtrsim \alpha \beta $. Then by Lemma \ref{Lemma:
Density delta set} with $\varepsilon =0$, there exists a finite $F\subset G$
such that $\left\vert F\right\vert \leq 1/{\alpha \beta }$ and $X\subseteq
F\cdot \mathcal{D}_{0}^{E}(X)$. If $g\in X$, there are $\xi \in F$ and $y\in
\mathcal{D}_{0}^{E}(X)$ such that $g=\xi y$. Since $y\in \mathcal{D}
_{0}^{E}(X)$, ${}^{\ast }\!A \cap \eta {}{}^{\ast }B ^{-1}\cap y{}^{\ast
}\!A \cap y\eta {}^{\ast}\!B^{-1}\neq \emptyset $. In particular, $y=ab\eta
^{-1}$ for some $a\in {}^{\ast }\!A$ and $b\in {}^{\ast }B$. Therefore,
\begin{equation*}
g=\xi y=\xi ab\eta ^{-1}
\end{equation*}
and
\begin{equation*}
g\eta =\xi ab\in F{}{}^{\ast}\!A {}^{\ast}\!B={}^{\ast }\!\left( FAB\right)
\text{.}
\end{equation*}
Since this is true for every $g\in X$,
\begin{equation*}
X\eta \subset {}^{\ast }\!\left( FAB\right) \text{.}
\end{equation*}
Hence, by the nonstandard characterization of finite embeddability,
\begin{equation*}
X\vartriangleleft FAB\text{.}\qedhere
\end{equation*}
\end{proof}

\medskip

\begin{corollary}
\label{Corollary: Jin with bound}Under the hypothesis of Theorem \ref%
{Theorem: Jin with bound}, $AB$ is piecewise $k$-syndetic where $%
k=\left\lfloor \frac{1}{\alpha \beta }\right\rfloor $.
\end{corollary}

\begin{proof}
Set $X=G$ and apply Theorem \ref{Theorem: Jin with bound}. Thus, $%
G\vartriangleleft FAB$ for some $F\subset G$ such that $\left\vert
F\right\vert \leq \frac{1}{\alpha \beta }$. Henceforth $FAB$ is thick and $%
AB $ is piecewise $\left\lfloor \frac{1}{\alpha \beta }\right\rfloor$%
-syndetic.
\end{proof}

\bigskip

\section{Countable amenable groups}

\label{sec-countable}

Throughout this section we focus on \emph{countable} amenable groups and
prove finite embeddability properties.

\smallskip By \cite{bbf} (Corollary 5.3), if $A\subseteq G$ has positive F\o %
lner density and $G$ is countable, then $AA^{-1}$ is piecewise Bohr.
Moreover, by \cite{bbf} (Lemma 5.4), if $A,B\subseteq G$, $A$ is piecewise
Bohr and $A\vartriangleleft B$, then $B$ is piecewise Bohr as well. It is a
standard result in ergodic theory (see for example \cite{lin}) that any
countable discrete amenable group $G$ has a F\o lner sequence $\left(
F_{n}\right) _{n\in \mathbb{N}}$ for which the \emph{pointwise ergodic
theorem} holds. This means that, if $G$ acts on a probability space $\left(
X,\mathcal{B},\mu \right) $ by measure preserving transformations and $f\in
L^{1}\left( \mu \right) $, then there is a $G$-invariant $\bar{f}\in
L^{1}\left( \mu \right) $ such that, for $\mu $-almost all $x\in X$:
\begin{equation*}
\lim_{n\to\infty} \frac{1}{\left\vert F_{n}\right\vert }\sum_{g\in
F_{n}}f\left( gx\right)\ =\ \bar{f}(x)\,.
\end{equation*}

\begin{lemma}
\label{Lemma: pullback} If $E$ is a F\o lner approximation of $G$, $0<\gamma
\leq 1$, and $C$ is an internal subset of $E$ such that $\frac{\left\vert
C\right\vert }{\left\vert E\right\vert }\gtrsim \gamma $, then there exists $%
\xi \in E$ such that
\begin{equation*}
d\left( C\xi ^{-1}\cap G\right) \geq \gamma \text{.}
\end{equation*}
\end{lemma}

\begin{proof}
Pick a F\o lner sequence $\left( F_{n}\right) _{n\in \mathbb{N}}$ for $G$
that satisfies the pointwise ergodic theorem. Consider the (separable) $%
\sigma $-algebra $\mathcal{B}$ on $E$ generated by the characteristic
function $\chi _{C}$ of $C$, the probability space $\left( E,\mathcal{B},\mu
\right) $ where $\mu $ is the restriction of the Loeb measure to $\mathcal{B}
$ and the measure preserving action of $G$ on $\left( E,\mathcal{B},\mu
\right) $ by left translations. Since $\chi _{C}$ belongs to $L^{1}\left(
\mu \right) $, there is a $G$-invariant function $\bar{f}\in L^{1}\left( \mu
\right) $ such that the sequence
\begin{equation*}
\left( \frac{1}{\left\vert F_{n}\right\vert }\sum_{g\in F_{n}}\chi
_{C}\left( gx\right) \right) _{n\in \mathbb{N}}
\end{equation*}%
converges to $\bar{f}\left( x\right) $ for $\mu $-a.a. $x\in E$ and hence,
by the Lebesgue dominated convergence theorem, in $L^{1}\left( \mu \right) $%
. This implies in particular that
\begin{equation*}
\int \bar{f}d\mu =\int \chi _{C}d\mu =\mathrm{st}\left( \frac{\left\vert
C\right\vert }{\left\vert E\right\vert }\right) =\gamma \text{.}
\end{equation*}%
Thus, the set of $x\in E$ such that $\bar{f}\left( x\right) \geq \gamma $ is
non negligible and, in particular, there is $\xi \in X$ such that $\bar{f}%
\left( \xi \right) \geq \gamma $ and the sequence
\begin{equation*}
\left( \frac{1}{\left\vert F_{n}\right\vert }\sum_{g\in F_{n}}\chi
_{C}\left( g\xi \right) \right) _{n\in \mathbb{N}}
\end{equation*}%
converges to $\bar{f}\left( \xi \right) \geq \gamma $. Observe now that, for
every $n\in \mathbb{N}$,
\begin{eqnarray*}
\frac{1}{\left\vert F_{n}\right\vert }\sum_{g\in F_{n}}\chi _{C}\left( g\xi
\right)  &=&\frac{\left\vert C\cap F_{n}\xi \right\vert }{\left\vert
F_{n}\right\vert } \\
&=&\frac{\left\vert C\xi ^{-1}\cap F_{n}\right\vert }{\left\vert
F_{n}\right\vert } \\
&=&\frac{\left\vert \left( C\xi ^{-1}\cap G\right) \cap F_{n}\right\vert }{%
\left\vert F_{n}\right\vert }\text{.}
\end{eqnarray*}%
From this and the fact $\left( F_{n}\right) _{n\in \mathbb{N}}$ is a F\o %
lner sequence for $G$ it follows that
\begin{equation*}
d\left( C\xi ^{-1}\cap G\right) \geq \gamma \text{.}\qedhere
\end{equation*}
\end{proof}

\medskip

\begin{theorem}
\label{Theorem: countable amenable} Let $G$ be a countable amenable group
and suppose that $A_{1},\ldots ,A_{n}\subseteq G$ have positive Banach
densities $d(A_{i})=\alpha _{i}$. Then there exists $B\subseteq G$ such that
$d(B)\geq \prod_{i=1}^{n}\alpha _{i}$ and $B\vartriangleleft A_{i}$ for
every $i=1,\ldots ,n$. As a consequence, $BB^{-1}\subseteq
\bigcap_{i=1}^{n}A_{i}A_{i}^{-1}$ and $\Delta _{\varepsilon }(B)\subseteq
\bigcap_{i=1}^{n}\Delta _{\varepsilon }(A_{i})$ for every $\varepsilon \geq
0 $. In particular, $\bigcap_{i=1}^{n}A_{i}A_{i}^{-1}$ is piecewise Bohr.
\end{theorem}

\begin{proof}
By Lemma \ref{concentration} there exists a F\o lner approximation $%
E\subseteq {}^{\ast }G$ and elements $\xi _{1},\ldots ,\xi _{n}\in {}^{\ast
}G$ such that
\begin{equation*}
\frac{\left\vert {}^{\ast }\!A_{0}\cap \bigcap_{i=1}^{n}{}^{\ast }\!A_{i}\xi
_{i}\cap E\right\vert }{\left\vert E\right\vert }\gtrsim
\prod_{i=0}^{n}\alpha _{i}\text{.}
\end{equation*}%
By applying \ref{Lemma: pullback} to $E$ and $C={}^{\ast }\!A_{0}\cap
\bigcap_{i=1}^{n}\xi _{i}{}^{\ast }\!A_{i}\cap E$ one obtains $\eta \in E$
such that
\begin{equation*}
d\left( C\eta ^{-1}\cap G\right) \geq \prod_{i=0}^{n}\alpha _{i}\text{.}
\end{equation*}%
Define $B=C\eta ^{-1}\cap G$ and observe that $B\eta \subseteq {}^{\ast
}\!A_{0}$ and $B\eta \xi _{i}^{-1}\subseteq {}^{\ast }\!A_{i}$ for $1\leq
i\leq n$. This implies that $B\vartriangleleft A_{i}$ for $0\leq i\leq n$.
\end{proof}

\bigskip

\section{Final remarks and open problems}

In a preliminary version of \cite{jin}, R. Jin asked whether one could
estimate the number $k$ needed to have $A+B+[0,k)$ thick (under the
assumption that both sets $A,B\subseteq\mathbb{N}$ have positive Banach
density). In the final published version of that paper, he then pointed out
that no such estimate for $k$ exists which depends only on the densities of $%
A$ and $B$. In fact, the following holds.

\medskip

\begin{itemize}
\item \emph{Let $\alpha,\beta>0$ be real numbers such that $\alpha+\beta<1$,
and let $k\in\mathbb{N}$. Then there exist sets $A_k,B_k\subseteq\mathbb{N}$
such that the asymptotic densities $d(A_k)>\alpha$ and $d(B_k)>\beta$ but $%
A_k+B_k+[0,k)$ is not thick. }
\end{itemize}

\smallskip An example can be constructed as follows.\footnote{%
~This example did not appear in \cite{jin}, and it is reproduced here with
Jin's permission.} Pick natural numbers $M,N,L$ such that $M/L>\alpha $ , $%
N/L>\beta $, and $M/L+N/L+1/L<1$. For every $k\in \mathbb{N}$, consider the
following subsets of $\mathbb{N}$:
\begin{equation*}
A_{k}\ =\ \bigcup_{n=0}^{\infty }[Lnk,Lnk+Mk)\quad \text{and}\quad B_{k}\ =\
\bigcup_{n=0}^{\infty }[Lnk,Lnk+Nk).
\end{equation*}
Then the following properties are verified in a straightforward manner:

\smallskip

\begin{itemize}
\item The asymptotic densities $d(A_{k})=M/L$ and $d(B_{k})=N/L$.

\smallskip

\item $A_{k}+B_{k}+[0,k)=\bigcup_{n=0}^{\infty }[Lnk,Lnk+Mk+Nk+k)$.
\end{itemize}

\smallskip

Since $Lkn+Mk+Nk+k<Lkn+Lk=Lk(n+1)$, it follows that $A_{k}+B_{k}+[0,k)$ is
not thick, as it consists of disjoint intervals of length $(M+N+1)k$.

\medskip

However, as remarked by M. Beiglb\"{o}ck, the problem was left open if one
replaces the length $k$ of the interval $[0,k)$ with the cardinality $k$ of
an arbitrary finite set. As shown by our Theorem \ref{Theorem: Jin with
bound}, one can in fact give the bound $k\leq 1/\alpha \beta $. Now, the
question naturally arises as to whether such a bound is optimal.

\bigskip

Next, it is easy to see that if $G$ is abelian and $B\subseteq G$ then $%
d\left( B\right) =d\left( B^{-1}\right) $. Thus, it follows from Corollary %
\ref{Corollary: Jin with bound} that if $A,B\subseteq G$ are such that $%
d\left( A\right) =\alpha $ and $d\left( B\right) =\beta $ and $G$ is abelian
then both $AB$ and $AB^{-1} $ are piecewise $\left\lfloor \frac{1}{\alpha
\beta } \right\rfloor $-syndetic. It would be interesting to know if the
same is true for more general amenable groups. More precisely: if $G$ is an
amenable group and $B\subseteq G$, then do $B$ and $B^{-1}$ have the same
density? Or at least, is it always the case that $B$ has positive density if
and only if $B^{-1}$ has positive density? Besides, is the statement of
Corollary \ref{Corollary: Jin with bound} still true where one replaces $AB$
with $AB^{-1}$?

\bigskip

Finally, all the results of this paper are proved without assumptions on the
cardinality of the group, apart from Theorem \ref{Theorem: countable
amenable}, where $G$ is supposed to be countable. It would be interesting to
know if also this result holds for any amenable group, regardless of its
cardinality.

\bigskip



\end{document}